\documentclass[reqno]{amsart}
\usepackage{xcolor}
\usepackage{booktabs}
\usepackage{siunitx}
\usepackage{chemmacros}
\usepackage{caption}
\usepackage{float}
\usepackage{soul}
\usepackage{amsmath}
\usepackage{amssymb}
\usepackage{amsthm}
\usepackage{verbatim}
\usepackage{amsrefs}
\usepackage{color}

\newtheorem{theorem}{Theorem}[section]

\newtheorem{corollary}[theorem]{Corollary}
\newtheorem{proposition}[theorem]{Proposition}

\theoremstyle{definition}
\newtheorem{definition}[theorem]{Definition}

\newtheorem{remark}[theorem]{Remark}

\numberwithin{equation}{section}
\numberwithin{equation}{section}

\title{ON LACUNARY STATISTICAL CONVERGENCE AND LACUNARY STRONG
CESÀRO CONVERGENCE BY MODULI}

\author{María del Pilar Romero de la Rosa}
\address{M.P. Romero de la Rosa, University \& Departamento de Matem\'aticas, Universidad de C\'adiz, Avenida de la Universidad,   11405    Jerez de la Frontera, C\'adiz (Spain)}

\email{pilar.romero@uca.es}
\date{November 2022}

\date{\today}


\begin{document}

\maketitle
\begin{abstract}
    Here we fully complete the studies initiated by Vinod K. Bhardwaj and Shweta Dhawan in \cite{hindawi} which relate different convergence methods which involves the classical statistical and the classical strong Cesàro convergences by means of lacunary sequences and measures of density in $\mathbb{N}$ modulated by a modulus function $f$.
\end{abstract}
\section{Introduction}
In Applied Science, sometimes, limit processes do no fit to the usual concept of limit. However, for applications, in most cases it is sufficient to consider a different form of convergence. This paradigm, appears in Mathematics for the first time with the study of the convergence of the Fourier series, sometimes they do not converge to a continuous function, but their partial sums always converge in the Cesàro sense.
This is precisely what is known as   summability methods or  convergence methods, an alternative formulation of convergence. Since then, a theory has been developed that has interest in its own right, moreover, this is a very active field of research with many contributors (see \cite{mursaleenbook}).

For instance, we say a sequence $(x_n)\subset X$  is said to be statistically convergent to $L$ if for any $\varepsilon$ the subset $\{n\,:\, \|x_n-L\|>\varepsilon\}$ has zero density on $\mathbb{N}$. On the other hand a sequence $(x_k)$ on a normed space $(X,\|\cdot\|) $ is said to be strong Ces\`aro convergent to $L$ if
$
\lim_{n\to \infty} \frac{1}{n}\sum_{k=1}^n \|x_k-L\|=0
$. These seems to be the first examples of convergence methods. 
The term statistical convergence was firstly presented by Fast \cite{Fast} and Steinhaus \cite{Stein} independently in the same year 1951.
And the strong  Ces\`aro convergence for real numbers was introduced by Hardy-Littlewood \cite{hardy} and Fekete \cite{fekete} in connection with the convergence of Fourier series.
It is a little bit surprising how the above  convergence methods, discovered in different times and which were developed independently, are basically the same concept thanks to a result discovered by Jeff Connor \cite{connor88}.

In many cases, applied problems required different convergence formulations and from a theoretical point of view it is very natural to ask if an alternative convergence method has really been introduced or on the contrary it is the same as the usual one. Or what relationship exists between a new method and the existing ones. In this sense, we see here that applications are a good source for good mathematical problems.

Strong convergence modulated by a modulus function $f$ ($f$-strong Cesàro convergence) were presented by Nakano \cite{nakano} in 1953. 
And independently after many years, Aizpuru  et al. \cite{aizpuru} presented a form of density on the set of  natural numbers, modulated also by a modulus function  $f$, from which can be defined  the concept of  $f$-statistical convergence. Since then, a great effort has been made to link the two types of convergence methods. One of the main contributions is due to V. K. Bhardwaj and S. Dhawan \cite{jia0}. In  \cite{jia} we fully solve several questions in connection with these convergence methods. For instance, we characterized  the modulus functions $f$ for which the $f$-statistically convergence and the statistically convergence are equivalents. Namelly, when $f$ is compatible (see Definition \ref{compatible} below).

More later, V. K. Bhardwaj and S. Dhawan \cite{hindawi}, 
 continue their work, relating both convergences: $f$ -statistical and $f$-strong Cesàro, in their lacunary versions, obtaining partial results. 
In this paper we solve some questions posed  by V. K. Bhardwaj and S. Dhawan \cite{hindawi}, Remark 12, p.4. Moreover, we fully characterize the relationship between the $f$-statistical convergence and the $f$-strong Cesàro convergence in their lacunary versions.

\section{Preliminaries}
\label{main}

In 1978 Freeman et al, inspired by the classical studies on Cesàro convergence by Hardy-Littlewood and Fekete, studied the concept of strong lacunary convergence in \cite{sember}.
By a lacunary sequence  $\theta=\{k_r\}_{r\geq 0}\subset \mathbb{N}$ with $k_0=0$, we will means   an increasing sequence of natural numbers such that $h_r=k_r-k_{r-1}\to \infty$ as $r\to \infty$.
Denoting the intervals $I_r=(k_{r-1},k_r]$,
a sequence $(x_n)$ on a normed space $X$ is said to be strong lacunary convergent to $L$ provided
$$
\lim_{r\to\infty}\frac{1}{h_r}\sum_{k\in I_r}^{\infty}\|x_k-L\|=0.
$$

Moreover,  they proved in \cite{sember} that this new convergence method is exactly equal to classical one introduced by Hardy-Littlewood and Fekete: strong Cesàro convergence, if and only if the ratios $q_r=\frac{k_r}{k_{r-1}}$ satisfy
$ 1 < \liminf_r q_r \leq  \limsup_r q_r < \infty$.

Later, Fridy and Orhan \cite{b5} refined these studies by relating them to a new concept: {\it lacunary statistical convergence}, where it underlies new ways of measuring the density of subsets in $\mathbb{N}$. Specifically, given a lacunary sequence $\theta=\{k_r\}$, a sequence $(x_n)$ in a normed space $X$ is said to be lacunary statistical convergent to $L$ if for any $\varepsilon >0$ the subset $A_\varepsilon=\{k\in\mathbb{N}\,:\,\|x_k-L\|>\varepsilon\}$ has 
$\theta$-density equal $0$. That is,
$$
d_\theta(A_\varepsilon)=\lim_{r\to\infty}\frac{1}{h_r}\{k\in I_r\,:\, \|x_k-L\|>\varepsilon\}=0.
$$

Let us recall that $f\,:\, \mathbb{R}^+\to \mathbb{R}^+$ is said to be a modulus function if it satisfies:
\begin{enumerate}
    \item $f(x)=0 $ if and only if $x=0$.
    \item $f(x+y)\leq f(x)+f(y)$ for every $x,y\in \mathbb{R}^+$.
    \item $f$ is increasing.
    \item $f$ is continuous from the right at $0$.
\end{enumerate}

To avoid trivialities, we will suppose that the modulus functions are unbounded. 

In the 1990s there has been a lot of interest in studying the classical Cesàro convergence, when it is modulated by a modulus function $f$. For instance, these studies started by Maddox \cite{maddox} and  continued later by Connor \cite{canadian}. 
If one want to find the "statistical" counterpart of the convergence methods studied by Maddox, we see that it is impossible \cite{jia}.  That is, we cannot  define a density in $\mathbb{N}$ that allows to study the convergence introduced by Maddox by means of
the statistical convergence defined by this  density.

However, many years later, Aizpuru et al. \cite{aizpuru}, introduced a concept of density for sets of $ \mathbb{N} $ modulated by a function $f$. Namelly, given $ A \subset \mathbb{N} $ the $ f $ -density is defined as:
$$
d_f(A)=\lim_{n\to \infty}\frac{1}{f(n)} f\left(\#\{k\leq n\,:\,k\in A\}\right)
$$
when this limit exists (here $\#C$ denotes the cardinal of a finite subset $C$).

As we show in \cite{jia}, a slight modification of the formula introduced by Madox gives us a concept of $f$-strong Cesàro convergence that fits as a glove to the $f$-density. And now it is possible to obtain the desired relationship between these two types of convergence methods.

A similar problem occurs with the lacunary version of the $f$-statistical convergence. Specifically, the lacunary $f$-statistical convergence introduced by Vinod K. Bhardwaj and Shweta Dhawanin \cite{hindawi} and a possible lacunary extension introduced by Pehlivan and Fisher \cite{fisher} for the $f$-strong Cesàro convergence.

The aim of this article is to explore the rich structure that exists between the lacunary statistical convergence and the strong cesàro convergence when are modulated by a modulus function $f$.
As it was defined in \cite{hindawi}, given a lacunary sequence $\theta=(k_r)$,  and a modulus function $f$, a sequence $(x_k)$ is said to be lacunary $f$-statistical convergent to $L$ if for any $\varepsilon>0$ the subset
$A_{\varepsilon}=\{k\in \mathbb{N}\,:\, \|x_k-L\|>\varepsilon\}$ has lacunary $f$-density equal to zero, that is
$$
d_{f,\theta}(A_\varepsilon)=\lim_{t\to\infty}\frac{1}{f(h_t)}f\left(\#\{k\in I_r, :\,\|z_k-L\|>\varepsilon\}\right)=0.
$$
Now, let us see that the densities have  two degrees of choice.
Next we introduce the slight modification of the lacunary $f$-strong Cesàro convergence studied by Pehlivan and Fisher.

\begin{definition}
Given a modulus function $f$ and $\theta=(k_r)$ a lacunary sequence, we say that a sequence $(x_n)$ is lacunary $f$-strong convergent to $L$ if
$$
\lim_{t\to\infty}\frac{1}{f(h_t)}f\left(\sum_{k\in I_r}\|x_k-L\|\right)=0.
$$
\end{definition}

In \cite{jia} we introduced the notion of compatible modulus function that will play a central role along the discussion.

\begin{definition}
\label{compatible}
Let us denote $\varphi(\varepsilon)=\limsup_n\frac{f(n\varepsilon)}{f(n)}$.
A modulus function  $f$ is said to be compatible if $\lim_{\varepsilon\to 0} \varphi(\varepsilon)=0$.
\end{definition}

\begin{remark}The following functions $f(x)=x^p+x^q$, $0<p,q\leq 1$, $f(x)=x^p+\log(x+1)$, $f(x)=x+\frac{x}{x+1}$ are modulus functions which are compatible. And $f(x)=\log(x+1)$, $f(x)=W(x)$ (where $W$ is the $W$-Lambert function restricted to $\mathbb{R}^+$, that is, the inverse of $xe^x$) are modulus functions which are not compatible. Indeed, let us show that $f(x)=x+\log(x+1)$ is compatible. 
$$
\lim_{n\to\infty}\frac{f(n\varepsilon')}{f(n)}=\lim_{n\to \infty}\frac{n\varepsilon'+\log(1+n\varepsilon')}{n+\log(n+1)}=\varepsilon'
$$
On the other hand  if $f(x)=\log(x+1)$, since
$$
\lim_{n\to\infty}\frac{\log(1+n\varepsilon')}{\log(1+n)}=1
$$
we obtain that $f(x)=\log(x+1)$ is not compatible.
\end{remark}

\begin{definition}
\label{thetacompatible}
Let us denote  $\varphi_{\theta}(\varepsilon)=\limsup_{t\to\infty}\frac{f(h_t\varepsilon)}{f(h_t)}$. We will say that  $f$ is $\theta$-compatible if  if $\lim_{\varepsilon\to 0}\varphi_\theta(\varepsilon)=0$. 
\end{definition}
Clearly, if $f$ is compatible then $f$ is $\theta$-compatible for any lacunary sequence $\theta$.

The paper is structured as follows. In Section \ref{secciontres} we fix a lacunary sequence $(\theta)$ and we will see the relationship between the space of sequences which are lacunary strong Cesàro convergent ($N_\theta$) and the space of lacunary $f$-strong Cesàro convergent sequences ($N_\theta^f$). 
Moreover, we see also the relationship between the spaces $S_\theta$ and $S_\theta^f$, that is, the space of lacunary statistical convergent sequences
and the space of lacunary $f$-statistical convergent sequences (respectively).

In Section \ref{cuatro} we will obtain Connor-Khan-Orhan's type results by relating the lacunary $f$ strong Cesàro convergence ($N_\theta^f$) with the lacunary $f$-statistical convergence ($S_\theta^f$).
Finally, in Section \ref{cinco} we fix a modulus function $f$ and we study the relationship between the lacunary $f$-strong Cesàro convergence and the $f$-strong Cesàro convergence, for any lacunary sequence $\theta$.

\section{Lacunary convergence vs Lacunary convergence modulated by a function $f$.}
\label{secciontres}
In this section we fix a lacunary sequence $\theta=(k_r)$ and we move the modulus functions $f$. We will see the relationship between lacunary $f$-strong Cesàro convergence and the lacunary strong Cesàro convergence.
Will we show that both convergence methods are equivalents when the modulus function $f$ is compatible.

A similar characterization is obtained for the lacunary $f$-statistical convergence and the lacunary statistical  convergence. If the modulus function $f$ is compatible both convergence methods are equivalents.

Moreover, we fully solve a question posed by V. K. Bhardwaj and S. Dhawan \cite{hindawi} Remark 12, p.4. If we denote by $S_\theta^f$ the space of all lacunary $f$-statistical convergence and by $S_\theta$ the set of all lacunary statistical convergence sequences, by applying Theorem \ref{propo1},  and Theorem \ref{reciproco} we have that $S_\theta^f=S_\theta$ if an only if $f$ {\it is $\theta$-compatible}.

Let $f$ be a modulus function and $\theta=(k_r)$ be a lacunary sequence.
Let us denote by $N^f_\theta$ the space of all lacunary $f$-strongly Ces\`aro convergent sequences and  by $N_\theta$ the space of  all lacunary strong Cesàro convergent sequences.

\begin{theorem}
\label{proposicion}
Let $f$ be a modulus function and let $\theta=(k_t)$ be a lacunary sequence. Then
$N_\theta^f\subset N_\theta$.
\end{theorem}
\begin{proof}
Indeed, for all $p\geq 1$, there exists $t_0$, such that:
$$
\frac{f\left(\sum_{k\in I_t}\|x_k-L\| \right)}{f(h_t)}\leq \frac{1}{p}
$$
for all $t\geq t_0$. That is,
$$
f\left(\sum_{k\in I_t}\|x_k-L\| \right)\leq \frac{1}{p}f(h_t) \leq f\left(\frac{h_t}{p}\right)
$$
and since $f$ is increasing, we have that $$
\sum_{k\in I_t}\|x_k-L\| \leq \frac{h_t}{p}
$$
for all $t\geq t_0$, that is, $(x_n)$ is lacunary strong Ces\`aro convergent to $L$.
\end{proof}

A similar result can be obtained for the lacunary statistical convergence.

\begin{theorem}
\label{propo1}
Let $f$ be a $\theta$-compatible modulus function and $\theta=(k_r)$ a lacunary sequence. Then $S_\theta= S_\theta^f$.
 \end{theorem}
\begin{proof}
It well know that for any modulus function $S_\theta^f\subset S_\theta$.
Let us fix $\varepsilon>0$ arbitrarily small. Since $f$ is $\theta$-compatible, there exist $\varepsilon' >0$  and $r_0(\varepsilon)$ such that if $h_r\geq r_0$ then
$$
\frac{f(h_r\varepsilon')}{f(h_r)}<\varepsilon.
$$

Let $c>0$ and let us fix $\varepsilon'>0$. Since $(x_n)$ is lacunary statistically convergent to $L$ then there exists $r_1(\varepsilon)$ (since $\varepsilon'$ depends  on $\varepsilon$ we have that $r_1$ depends actually on $\varepsilon$) such that if $h_r\geq r_1$:

$$
\#\{k\in I_r\,:\, \|x_k-L\|>c\}\leq h_{r}\varepsilon'.
$$

 Since $f$ is increasing we have
$$
\frac{f(\#\{ k\in I_r\,\,:\,\|x_k-L\|> c\})}{f(h_r)}\leq
\frac{f(h_r\varepsilon')}{f(h_r)}<\varepsilon.
$$
for all $h_t\geq \max\{r_0,r_1\}$, which gives the desired result. 
\end{proof}

\begin{theorem}
\label{propo2}
Let $f$ be a $\theta$-compatible modulus function and $\theta=(k_r)$ a lacunary sequence. Then, $N_f^\theta=N_\theta$. 
 \end{theorem}
\begin{proof} Let us suppose that $(x_n)$ is lacunary strong Ces\`aro convergent to $L$. Then, for any $\varepsilon'>0$ there exists $r_0$  such that if  $h_r\geq r_0$ then
$$
\sum_{k\in I_r} \|x_k-L\|\leq h_r \varepsilon'
$$
and since $f$ is increasing, then
$$
f\left(\sum_{k\in I_r} \|x_k-L\|\right)\leq f( h_r \varepsilon')
$$
thus
$$
\frac{f\left(\sum_{k\in I_r} \|x_k-L\|\right)}{f(h_r)}\leq \frac{f( h_r\varepsilon')}{f(h_r)}
$$
then by applying the same argument as above we get the desired result.
\end{proof}

\begin{theorem} Let $f$ be a modulus function and $\theta=(k_r)$ a lacunary sequence.
\label{reciproco}
\begin{enumerate}
    \item If $S_\theta=S^f_\theta$  then $f$ must be $\theta$-compatible.
    \item If  $N_\theta=N_\theta^f$  then $f$ must be $\theta$-compatible.
\end{enumerate}
\end{theorem}
\begin{proof}

Assume that $f$ is not a $\theta$-compatible modulus function.
Let us construct on the normed space $X=\mathbb{R}$ a sequence $(x_n)\in S_\theta\setminus S_\theta^f$. An easy modification provides a counterexample on any normed space. 

Since $\varphi_{\theta}(\varepsilon)$ is an increasing function, if $f$ is not $\theta$-compatible then  there exists $c>0$ such that for any $\varepsilon>0$:  $\varphi_\theta(\varepsilon)>c$.

Let us fix $\varepsilon_k\to 0$. Thus, for each $k$ there exists $h_{r_k}$ such that $f(h_{r_k}\varepsilon_k)\geq cf(h_{r_k})$. Moreover, since $h_{r}$ is increasing we can suppose that a
\begin{equation}
    \label{lacudesigualdad}
    h_{r_k}(1-\varepsilon_k)-1>0.
\end{equation} 
Let us denote by $\lfloor x\rfloor$ the integer part of $x$.  We set $n_k=\lfloor h_{r_k}\varepsilon_k \rfloor+1$. According to equation (\ref{lacudesigualdad}), we get that $h_{r_k}-n_k>0$. Let us define the subset $A_k=[k_{r_k}-n_k,k_{r_k}]\cap \mathbb{N}\subset I_{r_k}$, and $A=\bigcup_kA_k$. If we denote by $\chi_A(\cdot)$ the characteristic function of $A$, we claim that the sequence $x_n=\chi_A(n)$, is $\theta$-statistically convergent to $0$ but not $f$-statistically convergent, a contradiction.

Indeed, if $r\neq r_k$ for any $k$,   then
$$
 \frac{\#\{l\in I_{r}\,:\, |x_l|>\varepsilon\}}{h_{r}} \leq \frac{0}{h_{r}}=0.
$$
And for $r=r_k$:
$$
 \frac{\#\{l\in I_{r_k}\,:\, |x_l|>\varepsilon\}}{h_{r_k}}= \frac{n_k}{h_{r_k}}\to 0
$$
as $k\to \infty$. 

On the other hand, 
$$
 \frac{f(\#\{l\in I_{r_k}\,:\, |x_l|>\varepsilon\})}{f(h_{r_k})}= \frac{f(n_k)}{f(h_{r_k})}\geq \frac{f(h_{r_k}\varepsilon_k)}{f(h_{r_k})}\geq c
$$
which yields part (1).
To show part (2), let us see that if $f$ is not $\theta$-compatible the sequence $(x_n)$ constructed before satisfies that $(x_n)\in N_\theta$ but $(x_n)\notin N_\theta^f$ a contradiction.

\end{proof}

\section{Lacunary $f$-strong Cesàro convergence vs Lacunary $f$-statistical convergence}
\label{cuatro}
In this section we will obtain Connor-Khan-Orhan's type results by relating
the lacunary $f$ strong Cesàro convergence ($N_\theta^f$) with the lacunary $f$-statistical convergence ($S_\theta^f$).
\begin{theorem}
Assume that $\theta=(k_t)$ is a lacunary sequence and $f$ is a modulus function. Then $N_\theta^f\subset S_\theta^f$.
\end{theorem}
\begin{proof}
 First of all, in order to prove that $(x_n)$ is lacunary $f$-statistically convergent to $L$, it is sufficient to show that for all $m\in \mathbb{N}$,
\begin{equation}
    \label{ecuacion1}
\lim_{t\to \infty} \frac{f(\#\{k\in I_t\,\,:\,\, \|x_k-L\|>\frac{1}{m}\})}{f(h_t)} =0.
\end{equation}
Indeed, let us fix $\varepsilon>0$ and let us consider $m$ such that $\frac{1}{m+1}\leq \varepsilon\leq \frac{1}{m}$. Then, we get
$$
\#\{k\in I_t\, : \|x_k-L\|>\varepsilon \}\leq \# \left\{k\in I_t\,: \|x_k-L\|>\frac{1}{m+1} \right\},
$$
therefore, since $f$ is increasing
$$
\lim_{t\to \infty} \frac{f\left(\#\{k\in I_t\,: \|x_k-L\|>\varepsilon \}\right)}{f(h_t)} \leq  \lim_{t\to \infty}  \frac{f\left(\#\{k\in I_t\,: \|x_k-L\|>\frac{1}{m+1} \} \right)}{f(h_t)} 
$$
thus taking limits on $t\to \infty$ we obtain what we desired.

Therefore, let $m\in \mathbb{N}$, and let us show equation (\ref{ecuacion1}). 
\begin{align}
\label{des}
    f\left(\sum_{k\in I_t} \|x_k-L\|\right)    & \geq f\left(\sum_{\substack{k\in I_t \\ \|x_k-L\|\geq\frac{1}{m}}} \|x_k-L\|\right) \\
   & \geq f\left(\sum_{\substack{k\in I_t \\ \|x_k-L\|\geq\frac{1}{m}}} \frac{1}{m}\right)
    \geq \frac{1}{m} f\left(\sum_{\substack{k\in I_t \\ \|x_k-L\|\geq\frac{1}{m}}} 1 \right)\\
   &= \frac{1}{m} f\left(\#\left\{k\in I_t\,\,:\,\, \|x-L\|>\frac{1}{m}\right\}\right).
    \end{align}
Since $(x_n)$ is lacunary $f$-strongly Ces\`aro convergent to $L$, we have that
$$
\lim_{t\to \infty} \frac{f\left(\sum_{k\in I_t}\|x_k-L\| \right)}{f(h_t)}=0,
$$
therefore, dividing by $f(h_t)$ equation (\ref{des}), and taking limit as $t\to\infty$, we obtain that for each $m\in \mathbb{N}$ 
$$
\lim_{t\to \infty} \frac{f(\#\{k\in I_t\,\,:\,\, \|x_k-L\|>\frac{1}{m}\})}{f(h_t)} =0
$$
which implies that $(x_n)$  is lacunary $f$-statistically convergent to $L$
 as we desired.
\end{proof}

Statistical convergence implies strong Cesàro convergence in the realm of all bounded sequences. This is the 1988 Connor's result  \cite{connor88}. This result was sharpened by Khan and Orhan \cite{orhan} by proving that a sequence is statistically convergent if and only if it is Strong Cesàro convergent and uniformly integrable.

 In the following statement we can obtain a result analogous to that of Khan and Orhan, which is even new for the lacunary statistical convergence. Here, is crucial to measure optimally the integrability (in a lacunary form) of a sequence.

\begin{definition}
Let $\theta=(k_t)$ be a lacunary sequence. A sequence $(x_n)$ is said to be lacunary uniformly integrable if
$$
\lim_{M\to \infty} \sup_t \sum_{\substack{k\in I_t \\ \|x_k\|\geq M}}     \|x_k\|=0.
$$
\end{definition}
Let us denote by $I_\theta$ the space of all lacunary uniform integrable sequences.
Let us observe that if a sequence $(x_n)$ is bounded then $(x_n)$ is lacunary uniformly integrable, that is, $\ell_\infty(X)\subset I_\theta$. On the other hand, if a sequence $(x_n)$ is uniformly integrable then $(x_n-L)$ is also uniformly integrable for every $L\in X$. 

\begin{theorem}
\label{th3}
Assume that $\theta=(k_r)$ is a lacunary sequence
and   $f$ is a $\theta$-compatible modulus function. Then $S_\theta^f\cap I_\theta\subset N_\theta^f$.  Moreover, if for some modulus function $f$ we have  $S_\theta^f\cap I_\theta\subset N_\theta^f$ then the modulus $f$ must be $\theta$-compatible.
\end{theorem}
\begin{proof}
Assume that  $(x_n) $ is a  sequence  such that $(x_n) $ is   lacunary $f$-statistically convergent to $L$ and lacunary uniformly integrable.
Let us consider $\varepsilon > 0$.
Since $f$ is $\theta$-compatible there exists $\varepsilon'>0$ such that 
\begin{equation}
\label{desigualdad1}
\frac{f\left(h_t\varepsilon'\right)}{f(h_t)}<\frac{\varepsilon}{3}
\end{equation}
for all $t\geq t_0(\varepsilon)$. 

Now, since $(x_n)$ is lacunary uniformly integrable, there exists a natural number $M>0$ large enough  satisfying $\frac{1}{M}<\varepsilon'$ and for all $t\in\mathbb{N}$
\begin{equation}
    \label{desigualdad2}
  \frac{1}{h_t}  \sum_{\substack{k\in I_t \\ \|x_k-L\|\geq M}}\|x_k-L\|< \varepsilon'.
\end{equation}
And since $(x_n)$ is lacunary $f$-statistically convergent to $L$, there exists a natural number, which we abusively denote  by $t_0(\varepsilon)$, such that for all $t \geq t_0(\varepsilon)$:
\begin{equation}
\label{desigualdad3}
    \frac{1}{f(h_t)} f(\#\{ k\in I_t\,:\,\,\|x_k-L\|>\varepsilon'\})< \frac{\varepsilon}{3M}.
\end{equation}

Therefore
\begin{eqnarray}
\label{tres}
   \frac{f\left(\sum_{k\in I_t}\|x_k-L\| \right)}{f(h_t)}  \leq \frac{1}{f(h_t)} f\left( \sum_{\substack{k\in I_t \\ M> \|x_k-L\|\geq\varepsilon'}} \|x_k-L\| \right)+ \nonumber \\  \frac{1}{f(h_t)}  f\left( \sum_{\substack{k\in I_t \\ \|x_k-L\|\geq M }} \|x_k-L\| \right) + \nonumber \\
   \frac{1}{f(h_t)}  f\left( \sum_{\substack{k\in I_t \\ \|x_k-L\|<\varepsilon'}} \|x_k-L\| \right)
\end{eqnarray}

Since $f$ is increasing, according to (\ref{desigualdad3}) we get that for all $t\geq t_0(\varepsilon)$ the first term of (\ref{tres}) is:
\begin{align}
\label{d1}
    \frac{1}{f(h_t)} f\left( \sum_{\substack{k\in I_t \\ M> \|x_k-L\|\geq\varepsilon'}} \|x_k-L\| \right) &<   \frac{f\left( \#\{k\in I_t \,\,:\,\,\|x_k-L\|>\varepsilon'\}\cdot M\right)}{f(h_t)}  \nonumber\\
    &\leq M \frac{1}{f(h_t)} f\left( \#\{k\in I_t\,\,:\,\,\|x_k-L\|>\varepsilon'\} \right)\nonumber  \\
    &< M\frac{\varepsilon}{3M}=\frac{\varepsilon}{3}.
\end{align}

On the other hand, let us estimate the second summand of the inequality (\ref{tres}). Using that $f$ is increasing and by applying firstly the  inequality (\ref{desigualdad2}) and later inequality (\ref{desigualdad1}) we have that for $t\geq t_0(\varepsilon)$:

\begin{align}
\label{d2}
    \frac{1}{f(h_t)}  f\left( \sum_{\substack{k\in I_t \\ \|x_k-L\|\geq M }} \|x_k-L\| \right) &=\frac{1}{f(h_t)}  f\left( h_t\frac{1}{h_t}\sum_{\substack{k\in I_t \\ \|x_k-L\|\geq M }} \|x_k-L\| \right) \nonumber \\
    &\leq \frac{1}{f(h_t)}  f\left( h_t \varepsilon' \right) \leq \frac{\varepsilon}{3}. 
\end{align}

Finally, for the third summand in (\ref{tres}) by applying inequality (\ref{desigualdad1}) we obtain that if $t\geq t_0(\varepsilon)$:
\begin{equation}
    \label{d3}
    \frac{1}{f(h_t)} f\left( \sum_{\substack{k\in I_t \\ \|x_k-L\|\leq\varepsilon'}} \|x_k-L\| \right)
 \leq \frac{1}{f(h_t)}  f\left(h_t\frac{1}{M}  \right)<\frac{\varepsilon}{3}.
\end{equation}

Thus, by using   inequalities (\ref{d1}),(\ref{d2}), and (\ref{d3}) into inequality (\ref{tres}) we obtain that if $t\geq t_0(\varepsilon)$
$$
\frac{f\left(\sum_{k\in I_t}\|x_k-L\| \right)}{f(h_t)}  \leq \varepsilon,
$$
that is, $(x_n)$ is lacunary $f$-strong Ces\`aro convergent to $L$ as we desired.

For the converse, assume that $f$ is not $\theta$-compatible.
We can construct without loss a counterexample normed space $X=\mathbb{R}$. Since $f$ is not $\theta$-compatible given $(\varepsilon_k)$ a decreasing sequence converging to zero, there exists a subsequence $r_k$ such that $f(h_{r_k}\varepsilon_k)\geq cf(h_{r_k})$, for some $c>0$.
If we consider the sequence $x_n=\sum_{k=1}^\infty \varepsilon_k\chi_{(k_{r_k}-1,k_{r_k}]}(n)$. 
The sequence $(x_n)$ is clearly bounded and since $\varepsilon_k$ is decreasing an easy check show that $(x_n)$ is lacunary $f$-statistically convergent to $0$.
On the other hand:
$$
\frac{1}{f(r_k)}f\left(\sum_{n\in I_{r_k}} |x_n|\right)=\frac{f(h_{r_{k}}\varepsilon_k)}{f(h_{r_k})}\geq c
$$
which proves that $(x_n)$ is not lacunary $f$-strong Cesàro convergent, as we desired to prove.

\end{proof}

\section{Lacunary $f$-strong Cesàro convergence vs lacunary strong Cesàro convergence.}
\label{cinco}
Now, let us fix a modulus function $f$ and we study the relationship between the lacunary $f$-strong Cesàro convergence and the $f$-strong Cesàro convergence in terms of the properties of a lacunary sequence $\theta$. 
Let $f$ be a modulus function and $\theta=(k_r)$ a lacunary sequence. 
Let us denote by $N_\theta^f$ the space of all lacunary $f$-strong Cesàro convergent sequences in the normed space $X$. And we denote by $N^f$ the space of all $f$-strong Cesàro convergent sequences in the normed space $X$. When $f(x)=x$, then $N^x=N$ will denote the space of strong Cesàro convergent sequences.

\begin{proposition}
Let $f$ be a modulus function and let $\theta=(k_r)$ be a lacunary sequence.
\begin{description}
    \item[i)] If $1<\liminf q_r$ then $N^f\subset N_\theta^f$.
    \item[ii)] Assume that  $f$ is compatible.  If $N^f\subset N_\theta=N_\theta^f$ then $1<\liminf q_r$.
\end{description}
\end{proposition}
\begin{proof}
To prove i) assume that $\liminf q_r>1$. Thus there exists $\delta>0$ such that 
$1+\delta<q_r$ for all $r$.
Let $(x_n)\in N^f$, such that  $x_n\overset{N^f}{\longrightarrow}L$ for some $L\in X$, we will show that  $x_n\overset{N^f_\theta}{\longrightarrow}L$. Indeed, set 
$$
\tau_r=\frac{1}{f(h_r)}f\left(\sum_{k\in I_r} \|x_k-L\|\right).
$$
Then
\begin{eqnarray*}
    \tau_r&=&\frac{1}{f(h_r)}f\left(\sum_{k=k_{r-1}-1}^{k_r} \|x_k-L\|\right)\\
    &\leq & \frac{1}{f(h_r)}f\left(\sum_{k=1}^{k_r} \|x_k-L\|\right) +\frac{1}{f(h_r)}f\left(\sum_{k=1}^{k_{r-1}} \|x_k-L\|\right) \\
    &=&  \frac{f(k_r)}{f(h_r)}\frac{1}{f(k_r)}f\left(\sum_{k=1}^{k_r} \|x_k-L\|\right)+ \frac{f(k_{r-1})}{f(h_r)}\frac{1}{f(k_{r-1})}f\left(\sum_{k=1}^{k_{r-1}} \|x_k-L\|\right)\\
    &\leq & C \left[\frac{1}{f(k_r)}f\left(\sum_{k=1}^{k_r} \|x_k-L\|\right)+ \frac{1}{f(k_{r-1})}f\left(\sum_{k=1}^{k_{r-1}} \|x_k-L\|\right)  \right],
\end{eqnarray*}
where the last inequality follows from the fact that  $\frac{k_r}{h_r}\leq \frac{1+\delta}{\delta}$ and $\frac{k_{r-1}}{h_r}<\frac{1}{\delta}$ and using that $f$ is increasing, we get that $\frac{f(k_r)}{f(h_r)}$ and $\frac{f(k_{r-1})}{f(h_r)}$ are bounded. The result follows taking limits as $r\to \infty$.

We prove ii)  by  contradiction. Let us suppose  that $\liminf q_r=1$, and we will construct a sequence $(x_n)\in N^f$ such that $(x_n)\notin N^f_\theta$.
Indeed, since $\liminf q_r=1$ by Freedman-Sember's result, there exists $(x_n)\in N$ but $(x_n)\notin N_\theta$.
However, since $f$ is compatible, in particular $f$ is   $\theta$-compatible, therefore $N_\theta^f=N_\theta$. On the other hand, since $f$ is compatible $N=N^f$.  Hence, there exists $(x_n)\in N^f=N$ such that $(x_n)\notin N_\theta^f=N_\theta$ as we desired to prove.

\end{proof}

\begin{proposition}
Let $f$ be a  modulus function and let $\theta=(k_r)$  be a lacunary sequence. 
\begin{enumerate}
    \item Assume that $f$ is compatible. If $\limsup q_r<\infty$ then $N_\theta^f\subset N^f$.
    \item Assume that $f$ is $\theta$-compatible. If  $N_\theta^f\subset N^f$ then  $\limsup q_r<\infty$.
\end{enumerate}
\end{proposition}
\begin{proof}

Let us shown $i)$. Since $f$ is compatible, in particular $f$ is $\theta$-compatible, therefore $N_\theta^f=N_\theta$.
On the other hand, since $f$ is compatible $N=N^f$. Then, the result follows by applying Freemand-Sember's result (see \cite{sember}). That is, if $1<\liminf q_r$ then $N_\theta=N_\theta^f=\subset N=N^f$ as we desired.

We prove the necessity   by  contradiction. Let us suppose  that $\limsup q_r=\infty$. Following as in Lemma 2 of \cite{sember}, we can select a sequence $k_{r(j)}\subset \theta$ satisfying $q_{r(j)}>j$ and let us consider the sequence
$$
x_k=\begin{cases}
x_0\neq 0 & k_{r(j-1)}<k<2k_{r(j-1)} \quad j=1,2,3,\cdots\\
0 & \textrm{otherwise}.
\end{cases}
$$
We wish to prove that $(x_k)\in N_\theta^f$. Let us fix $\varepsilon>0$.
Since $f$ is $\theta$-compatible, there exists $\varepsilon'>0$ and $t_0$ such that if
$j\geq j_0$  then $\frac{f(h_{r(j)}\varepsilon')}{f(h_{r(j)})}<\varepsilon$.
Since $q(r(j))>j$, we get
$$
\frac{k_{r(j-1)}}{k_{r(j)}-k_{r(j-1)}}=\frac{1}{q(r_j)-1}<\frac{1}{j-1}.
$$
We consider $r(j)$ large enough such that $1/j-1<\varepsilon'$. Thus since $f$ is increasing for $r(j)\geq r(j_0)$ we get
$$
f(k_{r(j-1)})\leq f\left(\frac{1}{j-1}(k_{r(j)}-k_{r(j-1)})\right)<f(\varepsilon' h_{r(j)}).
$$
Dividing by $f(h_{r(j)})$ we get that for $r(j)\geq r(j_0)$
$$
\frac{f(k_{r(j-1)})}{f(h_{r(j)})}<\frac{f(\varepsilon' h_{r(j)})}{f(h_{r(j)})}<\varepsilon.
$$
Therefore for $r(j)\geq r(j_0)$ 
\begin{eqnarray*}
\frac{1}{f(h_{r(j)})}f\left(\sum_{k\in I_{r(j)} }\|x_k\|\right)&=& 
\frac{f(k_{r(j-1)})}{f(h_{r(j)})} \|x_0\|\\
&\leq &\varepsilon \|x_0\|
\end{eqnarray*}
which gives that $(x_k)\in N_\theta^f$.
Now let us see that $(x_k)\notin N^f$. Indeed, since the coordinates of $(x_k)$ are $x_0$ or $0$ the possibles strong limits for $(x_k)$ are $0$ and $x_0$. Let us see that $(x_k)\overset{N^f}{\nrightarrow}x_0$.
Indeed for $i=1,\cdots, k(r(j))$:
\begin{eqnarray*}
\frac{1}{ k_{r(j)}}\sum_{i=1}^{k_{r(j)}} \|x_i-x_0\|&\geq&  \frac{k_{r(j)}-2k_{r(j-1))}}{k_{r(j)}}\|x_0\|\\
&\geq & \left(1-\frac{2}{j}\right)\|x_0\|
\end{eqnarray*}
which does not converge to zero.
Now, let us see that $(x_k)\overset{N^f}{\nrightarrow}0$. Indeed, for $i=1,\cdots, 2k_{r(j-1)}$:
\begin{eqnarray*}
\frac{1}{2k_{r(j-1)}} \sum_{i=1}^{2k_{r(j-1)}} \|x_i\| &\geq &
\frac{k_{r(j-1)}}{2k_{r(j-1)}}\|x_0|\\
&\geq &\frac{\|x_0\|}{2}
\end{eqnarray*}
which does not converge to zero. Therefore $(x_k)\notin N^f$ which gives the desired result.
\end{proof}

\begin{corollary}
Let $f$ be a compatible modulus function and let $\theta=(k_r)$  be a lacunary sequence.
Then $N_\theta^f= N^f$ if and only if  $1<\liminf q_r\leq \limsup q_r<\infty$.
\end{corollary}

\section*{Availability of data and material}
Not applicable.

\section*{Acknowledgements}

\section*{Author's contribution}
The author contributed significantly to analysis and manuscript preparation and she helped perform the analysis with constructive discussions. 
\section*{Competing interests}
The author declare that they have no competing interests.

\section*{Funding}




\end{document}